\documentclass[12pt]{article}
\usepackage{amsmath,amssymb,amsthm, amsfonts,tcolorbox}
\usepackage{hyperref}
\hypersetup{
     colorlinks=true,
     linkcolor=black,
     filecolor=black,
     citecolor = black,      
     urlcolor=cyan,
     }

\usepackage{enumitem}
\usepackage[margin=1in]{geometry} 
\usepackage[nameinlink]{cleveref}
\usepackage{graphicx,color}

\theoremstyle{definition}
\usepackage{url}
\usepackage{dsfont} 
\usepackage{color}
\definecolor{red}{rgb}{1,0,0}

\definecolor{blue}{rgb}{0,0,1}

\definecolor{green}{rgb}{0,.6,0}

\usepackage{float}
\usepackage{lineno}
\usepackage{tabularx}
\usepackage{tikz,pgfplots}
\usetikzlibrary{decorations.pathreplacing,calligraphy}
\usetikzlibrary{shapes.geometric}
\usetikzlibrary{shapes.multipart}
\usetikzlibrary{arrows}

\usetikzlibrary{calc}
\usetikzlibrary{arrows.meta}
\usetikzlibrary{arrows}
\usetikzlibrary{patterns}
\usetikzlibrary{shapes,snakes}
\usetikzlibrary{decorations.markings}
\tikzset{vtx/.style={inner sep=1.7pt, outer sep=0pt, circle,fill,draw}, 
vtx red/.style={inner sep=2.7pt, outer sep=0pt, circle, fill=red,draw}, 
vtx blue/.style={inner sep=2.7pt, outer sep=0pt, circle, fill=blue,draw}, 
vtx green/.style={inner sep=2.7pt, outer sep=0pt, circle, fill=black!30!green,draw}, 
vtx gray/.style={inner sep=2.7pt, outer sep=0pt, circle, fill=gray!30!white,draw}, 
vtx3/.style={inner sep=1.2pt, outer sep=0pt, regular polygon, regular polygon sides=3, draw},
vtx4/.style={inner sep=1.7pt, outer sep=0pt, regular polygon, regular polygon sides=4, draw},
vtx5/.style={inner sep=1.7pt, outer sep=0pt, regular polygon, regular polygon sides=5, draw},
vtx_box/.style={inner sep=2pt, outer sep=0pt, rectangle, draw, fill=white},
blueish/.style={color=blue},
reddish/.style={color=red},
greenish/.style={color=black!30!green},
edge_blueish/.style={color=blue,line width=2pt},
edge_reddish/.style={color=red,line width=2pt,dashed},
edge_greenish/.style={color=black!30!green,line width=2pt, densely dotted},
edge_orange/.style={color=orange,line width=2pt},
directed edge/.style={decoration={markings, mark=at position 0.5 with {\arrow{Latex}}},  postaction={decorate}},
} 

\newtheorem{thm}{Theorem}[section]
\newtheorem{cor}[thm]{Corollary}
\newtheorem{lem}[thm]{Lemma}
\newtheorem{prop}[thm]{Proposition}

\newtheorem{obs}[thm]{Observation}

\newtheorem*{thm2.6}{Theorem 2.6}
\newtheorem*{thm3.4}{Theorem 3.4}
\newtheorem*{thm4.9}{Theorem 4.9}
\newtheorem*{thm4.19}{Theorem 4.19}

\theoremstyle{definition}
\newtheorem{rem}[thm]{Remark}

\theoremstyle{definition}

\theoremstyle{definition}

\newcommand{\fs}{\rightarrow}

\newcommand{\mr}{\operatorname{mr}}

\newcommand{\M}{\operatorname{M}}

\newcommand{\Z}{\operatorname{Z}}

\newcommand{\bit}{\begin{itemize}}
\newcommand{\eit}{\end{itemize}}
\newcommand{\ben}{\begin{enumerate}}
\newcommand{\een}{\end{enumerate}}
\newcommand{\beq}{\begin{equation}}
\newcommand{\eeq}{\end{equation}}
\newcommand{\bea}{\begin{eqnarray}} 
\newcommand{\eea}{\end{eqnarray}}
\newcommand{\bpf}{\begin{proof}}
\newcommand{\epf}{\end{proof}\ms}
\newcommand{\bmt}{\begin{bmatrix}}
\newcommand{\emt}{\end{bmatrix}}
\newcommand{\ms}{\medskip}

\usepackage{tcolorbox}

\newcommand{\Zp}{\operatorname{Z}_{+}}

\title{The $\Z_q$-forcing number for some graph families}

\author{Jorge Blanco\thanks{Dept.~of Mathematics, Yale  University, New Haven, CT, USA (jorge.blanco@yale.edu)}  \and Stephanie Einstein\thanks{Dept.~of Mathematics, Mount Holyoke College, South Hadley, MA, USA (einst22s@mtholyoke.edu)} \and Caleb Hostetler\thanks{Dept.~of Mathematics and Computer Science, Eastern Mennonite University, Harrisonburg, VA, USA (caleb.hostetler@emu.edu)}\and 
J\"urgen Kritschgau\thanks{Dept.~of Mathematics, Carnegie Mellon University, Pittsburgh, PA, USA (jkritsch@andrew.cmu.edu)} \and Daniel Ogbe\thanks{Dept.~of Mathematics, Massachusetts Institute of Technology, Cambrige, MA, USA (dogbe@mit.edu)} }

\date{January 31, 2023}

\begin{document}
\maketitle

\begin{abstract}
The zero forcing number was introduced as a combinatorial bound on the maximum nullity taken over the set of real symmetric matrices that respect the pattern of an underlying graph. 
The $\Z_q$-forcing game is an analog to the standard zero forcing game which incorporates inertia restrictions on the set of matrices associated with a graph. 
This work proves an upper bound on the $\Z_q$-forcing number for trees. 
Furthermore, we consider the $\Z_q$-forcing number for caterpillar cycles on $n$ vertices.
We focus on developing game theoretic proofs of upper and lower bounds.
\end{abstract}

\section{Introduction}

The Inverse Eigenvalue Problem for a graph (IEPG) $G$ on $n$ vertices is to determine all possible spectra  of matrices in 
\[\mathcal S(G)  =\{ A\in \mathbb R^{n\times n}: a_{i,j}=a_{j,i}, \text{ for $i\neq j$, $a_{i,j}\neq 0$ iff $ij\in E(G)$}\}.\]
In particular, the IEPG is about the spectra of real symmetric matrices, with free diagonal entries, and off diagonal entries that respect the pattern of a graph's edges.

There are a lot of spectral properties one might care about: rank, nullity, multiplicity lists, exact spectrum, spectral radius, spectral gap, etc. 
Let $\mr(G)$ denote the minimum rank across matrices in $\mathcal S(G)$, and let $\M(G)$ denote the maximum nullity. 
Notice that by dimension counting 
\[\M(G)+\mr(G)=n.\]
Therefore, determining the maximum nullity and the minimum rank are equivalent problems. 
Furthermore, since $\mathcal S(G)$ is closed under diagonal translations (ie, if $A\in \mathcal S(G)$, then $A-kI\in \mathcal S(G)$ for all $k\in \mathbb R$), the problems of maximum nullity and maximum multiplicity of an eigenvalue are the same. 
For these reasons, focusing on the maximum nullity of matrices in $\mathcal S(G)$ is a nice problem. 
For an introduction to the maximum nullity and minimum rank, see \cite{barioli2010zero}.

Zero forcing was originally introduced as a tool to bound $\M(G)$ by the AIM Minimum Rank and Special Graphs Work Group in \cite{rank2008special}.
The zero forcing process starts with an initial set of blue vertices in a graph, and colors additional vertices blue according to the following color change rule: 
Given a graph $G$ and a blue vertex $v\in V(G)$, if $v$ has one white neighbor $w$, then $v$ forces $w$ ($v$ colors $w$ blue). 
A \textit{zero forcing set} for $G$ is an initial set of blue vertices $B$ such that after iteratively and exhaustively applying the zero forcing color-change rule, every vertex in $G$ is blue.
The \textit{zero forcing number} of a graph is the size of a minimum zero forcing set, and is denoted $\Z(G)$. 
For an overview on the results for standard zero forcing, see \cite{hogben2022inverse}.

The pair $(p,q)$ is the \emph{inertia}  of a symmetric matrix $A$ if $A$ has $p$ positive eigenvalues and $q$ negative eigenvalues. 
Notice that if $A$ has inertia $(p,q)$, then $A$ has nullity $n-p-q$. 
The inertia set of a graph $\mathcal I(G)$ is the set of inertia that can be achieved by matrices in $\mathcal S(G)$.
It is known that if $(p,q)\in \mathcal I(G)$, then $(p',q')\in \mathcal I(G)$ where $p'\geq p, q'\geq q$ (see the Northeast Lemma in \cite{barrett2009inverse}). 
Therefore, the interesting question is to determine the minimal inertia in $\mathcal I(G)$. 

To tackle this question, we consider $\M_q(G)$, the maximum nullity of matrices in $A\in\mathcal S(G)$ such that $A$ has at most  $q$ negative eigenvalues.
There is a corresponding zero forcing game  which bounds $\M_q(G)$.

\begin{tcolorbox}[arc = 0mm,colback=white]
\textbf{$\Z_q$-Forcing Game:} (quoted almost exactly from \cite{BGH2015}) All the vertices of a graph $G$ are initially colored white and there are two players, known as Blue (who has tokens), and White. 
Blue will repeatedly apply one of the following three rules until all vertices are colored blue.

\begin{enumerate}[label=Rule\,\arabic*:, wide=0pt, leftmargin=*]
    \item For one token, Blue can change any vertex from white to blue.
    \item At no cost, if a blue vertex $v$ in $G$ has exactly one white neighbor $w$, then Blue may turn $w$ blue.
    \item Let the blue vertices be denoted by $B$, and $W_1,\dots, W_k$ be the vertex sets of the connected components of $G[V\setminus B]. $ Blue selects at least $q+1$ of the $W_i$ and announces the selection to White. White then will select a nonempty subset of these components, say $\{W_{i_1},\dots, W_{i_\ell}\}$ (with $\ell\geq 1$), and announces it back to Blue. At no cost, Blue can apply Rule 2 on $G[B\cup W_{i_1}\cup\cdots\cup W_{i_\ell}].$
\end{enumerate}
\end{tcolorbox}

The parameter $\Z_q(G)$ denotes the fewest number of tokens Blue must spend to color the whole graph $G$ blue with adversarial play by White.
Notice that $\Z_q$ truly is a two player game. 
The standard zero forcing ($\Z$, or $\Z_n$) and PSD zero forcing ($\Zp$, or $\Z_0$) are not usually phrased as games like $\Z_q$ since there are optimal strategies for Blue where all tokens are spent upfront at the beginning of the game (meaning that there is no dynamic interaction between the choices Blue and White make).

We will use the notation $v\fs w$ to indicate that a blue vertex $v$ turns a white vertex $w$ blue according to some forcing rule. 
All other  graph theoretic notation we use follows \cite{west2001introduction}.

Notice that Rule $3$ is more useful for White when $q$ is small.

\begin{prop}[Proposition 3.2 in \cite{BGH2015}]
For any graph $G$ on $n$ vertices,
\[\Zp=\Z_0\leq \Z_1\leq \cdots\leq \Z_n=\Z.\]
\end{prop}

Matrices $A\in \mathcal S(G)$ that realize $\M_q(G)$ give White a strategy for selecting components in Rule 3 of the $\Z_q$-forcing game. 
This insight leads to the following theorem.

\begin{thm}[Theorem 3.3 in \cite{BGH2015}]
For any graph $G$ and non-negative integer $q$, 
\[\M_q(G)\leq \Z_q(G).\]
\end{thm}

Butler et al. continued the study of $\Z_q$ in \cite{BEFHKLSWY20}.
In particular, \cite{BEFHKLSWY20} contains some results characterizing graphs $G$ with $\Z_q(G)$ equal to $1$ or $2$. 
The authors of \cite{BEFHKLSWY20} also find a way to compute $\Z_1(T)$ for trees via the following theorem:

\begin{thm}[Theorem 6 in \cite{BEFHKLSWY20}]\label{cite.thm6}
Let $T$ be a tree with $|V(T)|\geq 3$. Let $\mathcal P_S(v)$ denote the set of maximal paths in a tree $S$ with $v$ as one end point (i.e., paths with $v$ as an endpoint and not contained in a longer path). Then 
\[\Z_1(T)= 2 +\max_{v\in V(T)}\left(\max_{\substack{S\subseteq T\\v\in V(S)}}\left(\min_{P\in\mathcal P_S(v)}\left(\sum_{w\in V(P)} (\deg_S(w)-2)\right)\right)\right).\]
\end{thm}

In Section \ref{sec.tools}, we introduce some tools that we use throughout the paper. 
This section culminates in an analysis of $\Z_q$ under edge deletion, with the following result:

\begin{thm2.6}
Let $q\geq 0$ and $e\in E(G)$. Then 
\[\Z_q(G)-2\leq \Z_q(G-e)\leq \Z_q(G)+1.\]
\end{thm2.6}

We then move on to prove a general bound for trees in Section \ref{sec.treesForests}, given below.

\begin{thm3.4}
Let $T$ be a tree on $n\geq 2$ vertices.
Let $h=h(T,v)$ be the height of $T$ rooted at $v$, and $L_\ell$ is the set of at most $q$ vertices with the largest degrees at height $\ell$.
Then 
\[\Z_q(T)\leq  \min_{v \in V(T)}\left\{ \deg(v) +  \sum_{\ell = 1}^{h-1}\sum_{a\in L_\ell}\max\{0, \deg(a)-2\}\right\}.\]
\end{thm3.4}

Finally, we consider the $\Z_q$-forcing number for Caterpillar cycles in Section \ref{sec.cat}. 
We define $C_{n,k}$ to be some graph obtained from a cycle $C_n$ by adding at least $2$ and at most $k$ pendant vertices to each vertex in $C_n$. 
We define $P_{n,k}$ analogously. 
In Section $\ref{sec.cat}$, we prove the following bounds in the cases when $q=2,3$:

\begin{thm4.9}
 For $ k\geq 2$ and $n\geq 3$, we have \[\log_2(n)+1\leq \Z_2(C_{n,k}) \leq \log_2(n-2) + 2k + 1.\]
\end{thm4.9}

\begin{thm4.19}
 Let $n\geq 3$, $k,q\geq 2$, and $p= q+1+x$ where $x\equiv q+1\mod 2$. 
 Then \[2\log_3(2)\log_2(\tfrac{n}{200}) \leq \Z_{q}(C_{n,k}) \leq (p-2)\log_{2}(n) + p+ k-q+q(k-1).\]
\end{thm4.19}

Furthermore, we prove the upper bound in Theorem \ref{thm.corona.upper} for $q\geq 3$, which we believe is close to the truth.

\subsection{Enumeration Notation}
Throughout this paper, it will be convenient to enumerate rules, strategies, and phases of or concerning the $\Z_q$-forcing game and variations thereof. 
To help the reader keep in mind what each enumeration is enumerating we will use the following conventions: 
\begin{itemize}
\item Rules will be enumerated as Rule 1, Rule 2, etc. with alterations to these rules denoted by an asterisk as in Rule 3$^*$. 
\item Strategies for Blue and White will be enumerated as Option 1, Option 2, etc. with the intent that the player will use the first option available to them or that options are accompanied by mutually exclusive conditions that make it clear which option a player exercises during a game. 
\item We will describe the course of a $\Z_q$-forcing game through enumerated phases (Phase 1, Phase 2, etc.). 
This will be helpful in describing how  $\Z_q$-forcing games will play out while leaving the strategies for Blue and White implicit. 
\item Finally, we will enumerate how a strategy for a player given graph $G$ can be converted to a strategy for a graph $H$ by Adaptation 1, Adaptation 2, etc. 
This will be helpful when the strategy for a graph $G$ is not explicit, and we would like to translate some strategy for $G$ into a strategy for $H$, thereby relating their $\Z_q$ numbers.
\end{itemize}

\section{Preliminary Results}\label{sec.tools}

Let $G$ be a graph, and let $B\subseteq V(G)$. 
The minimum number of additional tokens Blue uses to color all vertices blue, continuing from the game state in which the vertices in $B$ are blue, is denoted $\Z_q(G,B)$.

\begin{obs}
For any graph $G$ and non-negative integer $q$ we have 
\[\Z_q(G)=\min_{B\subseteq V(G)}\{|B|+\Z_q(G,B)\}\]
and 
\[\Z(G,B)\leq \Z_q(G).\]
\end{obs}

\begin{obs}\label{edge_lemma}
For any graph $G$ and subset $B\subseteq V(G)$ of blue vertices, let $E_B$ be a subset of the edges between any two vertices in $B$. Then $\Z_q(G, B)=\Z_q(G- E_B, B)$ for any $q$.
\end{obs}

We refer to blue vertices with $2$ or more white neighbors as \textit{active vertices}. 
We say the \textit{activity} of a graph is the number of active vertices the graph has after Blue  has exhaustively used Rule 2.

\begin{prop}\label{tree.activity.0}
Let $T$ be a tree with non-empty blue set $B$. 
If the activity of $T$ is $0$, then $B$ is a zero forcing set of $T$. 
In particular, Blue can use initial set $B$ and Rule 2 to color all of $T$ blue.
\end{prop}

\begin{proof}
Assume that the activity of $T$ is $0$. 
For the sake of contradiction, suppose Rule 2 has been exhaustively applied, but $T$ still contains a white vertex.
This implies that there exists a white vertex $w$ with a blue neighbor $v$ (since $T$ is connected). 
Since the activity of $T$ is $0$, it follows that $v$ has exactly one white neighbor (namely $w$). 
This contradicts the assumption that Rule 2 has been exhaustively applied. 
\end{proof}

\begin{lem}\label{force_lemma}
Suppose that $G$ is a forest and $B\subseteq V(G)$ is a set of blue vertices. 
Let $b$ be the number of blue vertices with at least two white neighbors. 
If $b\geq q+1$, then Blue can use Rule 3 to perform a force. 
\end{lem}

\begin{proof}
Let $G$ be a forest.
Let  $A=\{a_1, a_2, \dots, a_b\}$ the set of $b$ active blue vertices.

For each maximal tree $T$ in $G$ that contains an active blue vertex, root $T$ at an arbitrary active blue vertex contained in $T$.
If $v$ is a blue vertex without a white child, then it has at most one white neighbor (its parent) and is not active.
Therefore, we know that each active blue vertex must have at least one white child vertex in its rooted tree. 
For each active vertex $a_i$ in $A$, Blue will choose one white component $W_i$ that contains a vertex that is a white child of $a_i$ to assign to $a_i$. 
No component will be assigned to multiple active vertices, as this would imply that $G$ contained a cycle, which is not possible since $G$ is a forest. 
Blue now has a set $W=\{W_1, W_2, ... , W_b\}$ of $b\geq q+1$ distinct white connected components to send to White.

White will return a subset of $W$, whose union we will call $S$. 
Let $A_S\subseteq A$ be the set of active blue vertices associated with the components returned by White, and let $a_m\in A_S$ be an active vertex with minimum distance to the root of the tree containing $a_m$.
In the graph $G[B\cup S]$, $a_m$ must have only one white child vertex due to the way components were assigned. 
Furthermore, $a_m$  cannot have a white parent in the set, or else $a_m$ would not have the minimum distance to it's root among the active vertices.
Therefore $a_m$ will have exactly one white neighbor in the  $G[B\cup S]$, and Blue will be able to force that white neighbor blue for free.
\end{proof}

When there are at least $q+1$ active vertices in a forest $G$, then Lemma~\ref{force_lemma} guarantees that Blue will get a free force. 
The details of how Blue and White pass white components back and forth determines which white vertex will turn blue. 
Essentially, this is the interesting part of the $\Z_q$-forcing game since Blue and White may not have the same preferences. 
By modifying the details of Rule 3, we will be able to apply Lemma~\ref{force_lemma} for both upper and lower bounds. 
In particular, once we appeal to Lemma~\ref{force_lemma} we can give White unilateral or nearly unilateral power to decide which vertex turns blue to derive upper bounds. 
Similarly, we can give Blue unilateral or nearly unilateral power to decide which vertex turns blue to derive lower bounds. 
This idea will appear in various forms throughout the proofs in this paper. 

The next theorem analyzes how $\Z_q$ behaves under edge deletion.

\begin{thm}\label{thm.edge.deletion}
Let $q\geq 0$ and $e\in E(G)$. Then 
\[\Z_q(G)-2\leq \Z_q(G-e)\leq \Z_q(G)+1.\]
\end{thm}

\begin{proof}
To prove the first inequality, notice that the presence or absences of edge $e$ does not matter after both $u$ and $v$ are blue. 
Therefore, Blue can use a strategy that guarantees at most $\Z_q(G-e)$ tokens are spent to turn $G-e$ blue after initially coloring the end points of $e$ blue. 
More formally, if $e=vw$, then  \[\Z_q(G)\leq \Z_q(G-e,\{v,w\})+2\leq \Z_q(G-e)+2. \]

To prove the second inequality, suppose that Blue has a strategy to turn $G$ blue in at most $\Z_q(G)$ tokens. 
Let $\{G_i=(G,B_i)\}$ be the family of possible game states that Blue's strategy allows, starting from $G_0 =(G,\varnothing)$. 
We will show that Blue can adapt this strategy for $G$ to a strategy for $G-e$ so that Blue will not spend more than $\Z_q(G)+1$ tokens to turn $G-e$ blue. 
Let $e=vw$. The adapted strategy is as follows:
\begin{enumerate}[label=Adaptation\,\arabic*:, wide=0pt, leftmargin=*]
\item If Blue would use Rule 1 to turn $u\notin\{v,w\}$ blue in game state $(G,B)$, then Blue will use Rule 1 to turn $u$ blue in game state $(G-e,B)$.
\item If Blue would use Rule 2 (possibly as a substep of Rule 3) with $x\fs y$ in game state $(G,B)$ and $xy\neq vw$, then Blue will use Rule 2 with $x\fs y$ in game state $(G-e,B)$.
\item If Blue would use Rule 2 (possibly as a substep of Rule 3) with $x\fs y$ in games state $(G,B)$ and $xy=vw$, then Blue will use Rule 1 to turn $y$ blue in game state $(G-e,B)$.
\item If Blue would use Rule 3 to pass $\{W_1,\dots, W_{q+1}\}$ to White in game state $(G,B)$ and $W_1,\dots, W_{q+1}$ are white components in game state $(G-e,B)$, then Blue will use Rule 3 to pass $\{W_1,\cdots, W_{q+1}\}$ to White in game state $(G-e,B)$.
\item If Blue would use Rule 3 to pass $\{W_1,\dots, W_{q+1}\}$ to White in game state $(G,B)$ and $W_1,\cdots, W_{q+1}$ are not white components in game state $(G-e,B)$, then without loss of generality, $V(W_{q+1})= V(W_v)\cup V(W_w)$ where $W_v$ and  $W_w$ are connected white components in $(G-e,B)$. 
In this case, Blue will use Rule 3 to pass either $\{W_1,\dots,W_q, W_{v}\}$ or $\{W_1,\dots,W_q, W_{w}\}$ to White in game state $(G-e,B)$, which ever option guarantees that Blue will be able to turn a vertex blue.
\end{enumerate}
The fact that Adaptations 1 through 4 work is evident. 
However, Adaptation 5 requires more thought. 
Suppose that Blue would use Rule 3 to pass $\{W_1,\dots, W_{q+1}\}$ to White in game state $(G,B)$ and $W_1,\cdots, W_{q+1}$ are not white components of in game state $(G-e,B)$. 
Only one white component from $G$ can be affected by deleting $e$.
Without loss of generality, suppose that $W_{q+1}$ is affected by deleting $e$, and therefore, $e\in E(W_{q+1})$ and $W_{q+1}$ is split into two components in $G-e$ where one component contains $v$ and the other contains $w$; Let these components be $W_w$ and $W_v$. 

It remains to be shown that Blue can pass one of $\{W_1,\dots,W_q, W_{v}\}$ or $\{W_1,\dots,W_q, W_{w}\}$ to White and obtain a free force. 
Suppose that Blue does not win a free force by passing $\{W_1,\dots,W_q, W_{v}, W_w\}$. 
This implies that White can pass back some set of components $W^v$ which (without loss of generality) contains $W_v$ but not $W_w$, and $W$ does not concede a force to Blue.
This implies that no vertex in $B$ has exactly one white neighbor in $G[B\cup \bigcup_{W\in W^v}W]$.
For the sake of contradiction, suppose that Blue does not win a free force by passing $\{W_1,\dots,W_q, W_{w}\}$ to White. 
This implies that White can pass back a set $W^w$ which contains $W_w$ and does not concede a force to Blue. 
Therefore, no vertex in $B$ has exactly one white neighbor in $G[B\cup \bigcup_{W\in W^w}W]$.
In this case, no vertex in $B$ has exactly one white neighbor in $G[B\cup \bigcup_{W\in W^v\cup W^w}W]$ which contradicts the fact that Blue can win a force by passing $\{W_1,\dots, W_{q+1}\}$ to White in game state $(G,B)$.
Thus, Blue can win a free force by passing $\{W_1,\dots,W_q, W_{w}\}$ to White in game state $(G-e,B)$.

To finish proving the upper bound, notice that Blue will spend at most $\Z(G)+1$ token to turn $G-e$ blue. 
This follows from the fact that Blue's adapted strategy will follow the strategy for $G$ almost exactly. 
In fact, the only adaptation that may demand that Blue spend a token in $G-e$ where a token was not spent in $G$ is Adaptation 3 (Adaptation 3 may be invoked as a substep of either Adaptation 4 or 5, depending on how White responds in Rule 3). 
However, Adaptation 3 can be invoked at most one time, since after the first time Adaptation 3 is invoked, both $v$ and $w$ will be blue. 
\end{proof}

Notice that the upper bound of Theorem \ref{thm.edge.deletion} is tight by inspecting the path graph $P_n$. 

\section{Trees and Forests of Stars}\label{sec.treesForests}

First, we will consider how $\Z_q$ behaves on forests of stars. 
This will get the reader acquainted to our notion of active vertices, which represent an opportunity for Blue to use the novel Rule 3. 
After discussing stars, we move to proving an upper bound for $\Z_q(T)$ where $T$ is a tree. 
\subsection{Stars}
Let $S_n$ denote the star graph on $n$ vertices.

\begin{prop}\label{prop.stars}
Let $G$ be a forest graph of $k$ stars, where $ G = \bigcup_{i=1}^{k}S_{n_i}$ such that $n_1\geq n_2 \geq \ldots \geq n_k$. 
Then,\\ 
\[\Z_q(G) = \begin{cases}
k-q + \sum_{i = 1}^{q}\Z(S_{n_i}) & q < k \\
\sum_{i =1}^{k}\Z(S_{n_i}) & q \geq k
\end{cases}
.\]
\end{prop}
\begin{proof}
Let $G$ be a forest graph with $k$ stars as defined in the proposition.
To prove the equality, we will prove that the right side of the statement is both an upper bound and a lower bound for $\Z_q(G)$.

We will establish an upper bound by describing a game that is possible under any play by White. 

\begin{enumerate}[label=Phase\,\arabic*:, wide=0pt, leftmargin=*]
    \item First, Blue spends 1 token on an arbitrary leaf of each of the $k$ stars, and then uses Rule 2 to force the adjacent center of each of these leaves blue.
    \item Next, there will be some number of stars that have blue centers with multiple white neighbors. 
    As a result of Lemma \ref{force_lemma}, Blue will be able to repeatedly use Rule 3 to force vertices blue until there are $q$ or fewer blue vertices with two or more white neighbors. After this, there will be $m=\min\{k, q\}$ stars that still have white vertices. 
    \item We will consider the worst case scenario in which case the remaining stars are the $m$ stars with the most vertices ($S_{n_1}, S_{n_2}, ... S_{n_m}$) and that the only blue vertices on each of these stars are the two that were colored blue in the first step of the game.
    For each of the remaining stars, Blue spends $\Z(S_{n_i})-1$ tokens to turn $S_{n_i}$ blue, since one token was already spent on an arbitrary leaf.
\end{enumerate}
After all of this, Blue will have spent $k$ tokens at the start of the game, and at most $Z(S_{n_i})-1$ tokens on each of $m$ stars remaining after the repeated use of Rule 3.
Thus, we have an upper bound: 
\[\Z_q(G) \leq \begin{cases}
k + \sum_{i = 1}^{q}\Z(S_{n_i})-1 & q < k \\
k + \sum_{i =1}^{k}\Z(S_{n_i})-1 & q \geq k
\end{cases}
.\]

Now, we will establish a lower bound by describing a strategy for White.
Notice that any strategy for White must outline how White will select white components to pass back to Blue when Rule 3 is invoked.
Suppose Blue passes $W=\{W_1,\dots,W_{q+1}\}$ to White, and let $m= \max\{k,q\}$.
Here are the options White has in this strategy:

\begin{enumerate}[label=Option\,\arabic*:, wide=0pt, leftmargin=*]
    \item If all components in $W$ are contained in $\bigcup_{1\leq i \leq m} V(S_{n_i})$, then by the Pigeon Hole Principle, some $V(S_{n_i})$ contains at least two components. 
    White will pass these two components back to Blue.
    \item  If some component $W_i$ is not contained in $\bigcup_{1\leq i \leq m} V(S_{n_i})$, then White will pass back $W_i$.

\end{enumerate}
When White uses this strategy, Blue will be unable to use Rule 3 to turn vertices inside $\bigcup_{1\leq i \leq m} V(S_{n_i})$ blue. 
Therefore, Blue will have to spend $Z(S_{n_i})$ on each of the stars $S_{n_1},\dots, S_{n_{m}}$.
Furthermore, Blue must spend at least one token to turn a vertex in $S_{n_i}$ blue for each $i\geq m$. 
Therefore, when White uses this strategy, White will be able to force Blue to spend:
\[\Z_q(G) \geq \begin{cases}
k-q + \sum_{i = 1}^{q}\Z(S_{n_i}) & q < k \\
\sum_{i =1}^{k}\Z(S_{n_i}) & q \geq k
\end{cases}
.\]
Since we have matching upper and lower bounds, the proof is done.
\end{proof}

\begin{rem}
Notice that the lower bound in Proposition \ref{prop.stars} can also be obtained by considering matrices with an appropriate number of negative eigenvalues. 
Let $A_{n}$ be the adjacency matrix of $S_n$ (star with $n$ leaves), and let $\oplus$ denote the direct sum of matrices. 
When $q<k$, notice that
\[\bigoplus_{1\leq i \leq q}A_{n_i}\oplus \bigoplus_{q<i\leq k} 
(A_{n_i}+\sqrt{n_i} I_{n_i})\]
has $q$ negative eigenvalues and nullity $\sum_{1\leq i \leq q} (n_i-1)+k-q$.
When $q\geq k$, notice that 
\[\bigoplus_{1\leq i \leq k} A_{n_i}\]
has $k$ negative eigenvalues, and nullity $\sum_{1\leq i \leq k} (n_i-1)$. 
Furthermore, both of these examples show that $\M_q$ and $\Z_q$ agree on forests of stars.
\end{rem}

\subsection{Trees} \label{sec.trees}

We consider a  variant of the $\Z_q$-forcing game, which we denote as $\Z_q^* $, where the following rule replaces the standard third rule:  
\begin{enumerate}[label=Rule\,\arabic*$^*$:, wide=0pt, leftmargin=*]
\setcounter{enumi}{2}
    \item Whenever the activity of the graph is at least $q+1$, Blue can force  White to choose blue-adjacent white vertices to color blue for free until the activity of the graph is at most $q$.
\end{enumerate}
We refer to this rule as the \textit{deactivation rule}. 

Note that because of Lemma \ref{force_lemma}, the deactivation rule acts as a weaker version of the original Rule 3 for Blue.
If Blue has at least $q+1$ active vertices, then Blue would be able to perform forces for free until the activity of the graph decreases to at most $q$.
The deactivation rule simply lets White decide which vertices will turn blue.
Hence, an upper bound on $\Z^* _q$ is an upper bound on $
\Z_q$.

\begin{lem}\label{lem.force.height}
Suppose $G$ is a forest with $k$ components such that each component has  exactly $1$ blue vertex. Let the blue vertices in $G$ be $B=\{b_1,b_2,\dots, b_k\}$ with $\deg(b_1)\geq \deg(b_2)\geq \cdots\geq \deg(b_k)$. 
Then Blue can turn the neighborhood of $B$ blue for at most \[\sum_{1\leq i \leq m} \deg(b_i)-1\] tokens where $m=\max\{k,q\}$.
\end{lem}

\begin{proof}
Blue can use Rule 3 to turn a neighbor of $B$ blue as long as there are at least $q+1$ active vertices in $B$. 
Once there are at most $q$ active vertices in $G$, Blue can spend at most \[\sum_{1\leq i \leq m} \deg(b_i)-1\] tokens to deactivate all the active vertices in $B$.
Finally, when there are no active vertices in $B$, Blue can use Rule 2 to turn any white neighbors of $B$ blue.
\end{proof}

If $T$ is a tree at rooted at $v$, we say that a vertex $u$ is at height $i$ where $i$ is the distance from $u$ to $v$. 
The height of $T$ rooted at $v$ is the maximum height of any vertex in $T$.

\begin{thm}\label{thm.tree.upper}
Let $T$ be a tree on $n\geq 2$ vertices.
Let $h=h(T,v)$ be the height of $T$ rooted at $v$, and $L_\ell$ is the set of at most $q$ vertices with the largest degrees at height $\ell$.
Then 
\[\Z_q(T)\leq  \min_{v \in V(T)}\left\{ \deg(v) +  \sum_{\ell = 1}^{h-1}\sum_{a\in L_\ell}\max\{0, \deg(a)-2\}\right\}.\]
\end{thm}

\begin{proof}

To prove the desired upper bound, we exhibit a strategy for Blue in the $\Z^*_q$-forcing game.
Suppose $T$ is rooted at vertex $v$.
Consider the following strategy for Blue:
Blue will only consider the lowest height $\ell$ which contains a white vertex. 
\begin{enumerate}[label=Option\,\arabic*:, wide=0pt, leftmargin=*]
\item If $v$ is white, Blue will spend a token to turn $v$ blue.
\item If $v$ has $d$ white neighbors, Blue will spend $d-1$ tokens to turn all neighbors of $v$ blue.
\item If height $\ell\geq 2$ has a white vertex, then Blue will spend at most \[\sum_{a\in L_{\ell-1}}\max\{0, \deg(a)-2\}\] tokens to turn all vertices at height $\ell$ blue. 
This can be done by Lemma~\ref{lem.force.height}.
\end{enumerate}
Options 1 and 2 will cost Blue a combined $\deg(v)$ tokens.
Furthermore, Option 3 will be exercised at most once for each height $2\leq i\leq h$.
 Hence, 
\[
\Z_q^* (T) \leq \deg(v) + \sum\limits_{l=1}^{h-1} \sum\limits_{a\in L_{\ell} } \mathop{\max} \{ 0,\deg(a)-2\}
.\]
As $v$ was an arbitrary root, the theorem is established. 
\end{proof}

Notice that Theorem \ref{thm.tree.upper} appears to be tight for complete $k$-ary trees where $k\geq q+1$. 
The main feature of a complete $k$-ary tree is that there is a root which minimizes the height of the tree, and each level of the rooted tree contains at least $q$ vertices. 
This suggests that ``path like" trees that do not have a nice root may have a lower $\Z_q$-forcing number.

\section{Caterpillar Cycles}\label{sec.cat}

Recall that we defined $C_{n,k}$ to be some graph obtained from a cycle $C_n$ by adding at least $2$ and at most $k$ pendant vertices to each vertex in $C_n$. 
We refer to the  vertices in $C_{n,k}$ with degree greater than one as \emph{centers}.

The \emph{corona of two graphs} $G_1\circ G_2$ is obtained by creating $|V(G_1)|$ copies of $G_2$ and joining the $i^{th}$ copy of $G_2$ to the $i^{th}$ vertex of $G_1$. In the case where $G_2=kK_1$, we call $G_1\circ kK_1$ a $k$-corona of $G_1$. 
The corona operation appears in the study of skew zero forcing, were it is useful for characterizing low skew throttling numbers \cite{curl2020skew}. 
Our graph $C_{n,k}$ can be thought of as an irregular corona between $C_n$ and $K_1$.
In fact, we encourage the reader to treat $C_{n,k}$ as $C_n\circ kK_1$ on the first pass.

Our interest in $C_{n,k}$ is motivated by a few ideas. 
First, caterpillar graphs (of which $P_n\circ kK_1$ is an example) seem to have a much lower $q$-forcing number than the bound in Theorem \ref{thm.tree.upper}. 
Second, Blue's first few moves are easier to analyze on $C_{n,k}$ than $P_{n,k}$ because  $C_{n,k}$ does not have end points. 
Since $C_{n,k}$ and $P_{n,k}$ differ by a single edge, we don't mind ``relaxing" our study of trees to a simple uni-cyclic graph. 
Finally, $q$-forcing has not been studied for uni-cyclic graphs.

\begin{obs}
The only vertices in $C_{n,k}$ that can be active are the center vertices.
\end{obs}

\begin{prop}
For $C_{n}\circ kK_1$ with $n\geq 3$ and $k\geq 1$,
 \[\Z_1(C_{n}\circ kK_1) = k+1.\]
\end{prop}

\begin{proof}
Let $c_0,\dots, c_{n-1}$ denote the center vertices on the cycle in $C_{n}\circ kK_1$. To prove the upper bound, consider the following strategy for Blue:
\begin{enumerate}[label=Option\,\arabic*:, wide=0pt, leftmargin=*]
\item Apply Rule 2, if possible.
\item Apply Lemma \ref{force_lemma} and Rule 3, if possible. 
\item If there are fewer than 2 active vertices, then spend a token to turn a degree $1$ vertex with a white neighbor blue. 
\end{enumerate}

Clearly, the game will open with Blue spending $2$ tokens on leaves, which will then force their respective centers using Rule 2. 
We claim that the game will proceed to a state in which all center vertices are blue and at most one center vertex has white neighbors without Blue spending more tokens.

To prove the claim, suppose that Options $1$ and $2$ cannot be invoked. 
Since Option 1 cannot be invoked, all blue centers have $0$ or at least $2$ white neighbors.
Since Option 2 cannot be invoked, at most one blue center vertex has at least $2$ white neighbors. 

For the sake of contradiction, suppose that there is a white center vertex $c_0$.
Since the graph has at least two blue centers, there are distinct indices $i,j$ such that $c_i,c_j$ are blue and $c_k$ is white for $0< k < i$ and $j<k <n$.
Because both $c_i$ and $c_j$ have at least $1$ white neighbor and all blue centers have $0$ or $2$ white neighbors, it follows that $c_i$ and $c_j$ each have at least $2$ white neighbors. 
However, we have already stated that at most one blue center vertex has at least $2$ white neighbors, so we have arrived at a contradiction. 
Therefore, we will proceed under the assumption that all center vertices are blue, and at most one center vertex has white neighbors of degree $1$. 

Without loss of generality, suppose $c_0$ is the only center vertex with white neighbors. 
In the worst case scenario, $c_0$ has $k$ white neighbors, and Blue will be forced to spend $k-1$ tokens to finish the game. 
This gives $\Z_1(C_{n}\circ kK_1)\leq k+1$. 

We will now turn our attention to the lower bound. 
Observe that any iteration of the $Z_1$-forcing game on $C_{n}\circ kK_1$ will start with Blue spending at least $2$ tokens in the first three turns. 
After Blue spends their second token, White will use the following strategy: Select a white center with no blue degree $1$ neighbors, say $c_0$. 
If Blue invokes Rule 3 and passes $W=\{W_1,W_2\}$ to White, then White has three options in their strategy.
\begin{enumerate}[label=Option\,\arabic*:, wide=0pt, leftmargin=*]
\item If $c_0$ is white, then without loss of generality, $c_0\notin W_2$ and White will pass $W_2$ back to Blue.
\item If $c_0$ is blue where both $W_1$ and $W_2$ consist of degree $1$ vertices adjacent to $c_0$, then White will pass back both $W_1$ and $W_2$.
\item If $c_0$ is blue and, without loss of generality, $W_1$ does not contain any degree $1$ neighbors of $c_0$, then White will pass $W_1$ back to Blue. 
\end{enumerate}
Notice that as long as White follows this strategy, Blue will not be able to color any of the degree $1$ neighbors of $c_0$ blue for free without spending at least $k-1$ tokens on degree $1$ neighbors of $c_0$. 
Therefore, Blue must spend at least $k+1$ tokens. This completes the proof. 
\end{proof}

\subsection{The $q=2$ case}
Throughout this section, we will assume that $k\geq 2$. 
 Suppose that $G$ is a graph isomorphic to $C_{n,k}$ with some number of blue vertices. 
 We say that a center $v$ is \emph{bad} if $v$ is blue or $v$ has at least one blue neighbor of degree $1$.
 Centers which are not bad are called \emph{good}. 
 We say a path $P$ of centers $c_1c_2\dots c_m$ is \emph{protected} if every center in $P$ has at least one white neighbor of degree $1$ and $c_1$ and $c_m$ are the only two possible bad centers in $P$. 
 We will allow $c_1=c_m$ when $P$ is the full cycle.
 
 Let $\varphi(G,P)=m-2$ be the number of guaranteed good centers in $P$. 
 Define $\varphi(G)=\max_{P}\{\varphi(G,P)\}$ where $P$ ranges over the protected paths in $G$. 
 
 We will derive a lower bound for $\Z_2(C_{n,k})$ by tracking $\varphi$ as the forcing game unfolds. 
 The key idea is that White has a strategy which forces Blue to spend tokens to break up  protected paths. 
 Here is the strategy we will analyze for White: 
 If $\varphi(G)>0$, White will select a protected path of centers $P$ that realizes $\varphi(G)$. 
 Whenever Blue uses Rule $3$ and passes a set of white components $\{W_1,W_2, W_3\}$ to White, then White will do one of two things:
 \begin{enumerate}[label=Option\,\arabic*:, wide=0pt, leftmargin=*]
\item White basses back a component $W_i$ such that $W_i$ does not contain a vertex in $P$ or a degree $1$ vertex adjacent to a vertex in $P$. 
\item If every component in $\{W_1,W_2, W_3\}$ contains a vertex in $P$ or a degree one vertex adjacent to a vertex in $P$, then White passes back a subset $S\subseteq \{W_1,W_2, W_3\}$ such that Rule 2 cannot be used on $G[B\cup_{W_i\in S} W_i]$. 
\end{enumerate}
As the name suggests, protected paths should be ``protectable" by White.

\textbf{Claim:} If every component in $W=\{W_1,W_2,W_3\}$ contains a vertex in a protected path $P$ or a degree one vertex adjacent to a vertex in $P$, then there exists a subset $S\subseteq \{W_1,W_2,W_3\}$ such that Rule 2 cannot be used on $G[B\cup_{W_i\in S} W_i]$.

\begin{proof}
Suppose $P=c_1\dots c_m$ is a protected path in $G$. 
Without loss of generality, if $W_1$ and $W_2$ are isolated vertices adjacent to the same bad blue center in $P$, then using $S=\{W_1,W_2\}$ proves the claim. 
Therefore, we will assume that no two components in $W$ are isolated vertices adjacent to the same bad blue center of $P$. 
Therefore, without loss of generality, $W_1$ is an isolated white vertex adjacent to $c_1$, $W_2$ is an isolated white vertex adjacent to $c_m$, and $W_3$ contains $c_2,\dots, c_{m-1}$ and their white neighbors. 
In this case, White can pass back $W$ to Blue without conceeding a force, proving the claim. 
\end{proof}
 
Suppose $\{H_i\}^t_{i=0}$ is a sequence of colored copies of $C_{n,k}$ where $H_0$ has all white vertices, $H_t$ has all blue vertices, and $H_i$ follows from $H_{i-1}$ by one application of some $Z_2$ game rule. 
In particular, assume that White plays according to the strategy above in the sequence $\{H_i\}_{i=1}^t$. 
We will also assume that Blue successfully turns a vertex blue with every move so that $H_i$ has exactly one more blue vertex than $H_{i-1}$.

\begin{lem}\label{lem.nodecrease}
If Blue does not spend a token to transition from $H_{i-1}$ to $H_i$, then $\varphi(H_i)=\varphi(H_{i-1})$.
\end{lem}

The proof of Lemma \ref{lem.noprogress} in Section \ref{sec:q=3} is similar and more detailed.

\begin{proof}
Assume that $\varphi(H_{i-1})>0$ since otherwise the claim is trivially true. 
Suppose $P=c_1\cdots c_m$ is a protected path in $H_{i-1}$ that realizes $\varphi(H_{i-1})$. 
Notice that Rule 2 cannot be used to turn a good vertex in $P$ blue (other than $c_1,c_m$), nor can Rule 2 to be used to turn a degree 1 neighbor of a vertex in $P$ blue, since the only possible blue vertices in $P$ are  $c_1,c_m$, or degree 1 neighbors of $c_1,c_m$.

Similarly, the claim above shows that Rule $3$ cannot be used to turn a vertex in $P$ or any degree $1$ neighbor of a vertex in $P$ blue. 
This completes the proof. 
\end{proof}

\begin{lem}\label{lem.decrease}
If Blue spends a token to transition from $H_{i-1}$ to $H_i$, then \[\varphi(H_i)\geq \frac{\varphi(H_{i-1})-1}{2}.\]
\end{lem}

\begin{proof}
Suppose that $P=c_1\dots c_m$ is a protected path in $H_{i-1}$ which realizes $\varphi(H_{i-1})$. 
If $P$ is protected in $H_i$, then we are done. 
Therefore, we will assume that $P$ is not protected in $H_i$. 
In particular, it must be the case that $c_x$ is a bad center in $H_i$ for some $1<x<m$. 
Then $c_1\cdots c_x$ and $c_x\cdots c_m$ are protected paths in $H_i$, the longer of which has at least $\varphi(H_i)\geq \tfrac{\varphi(H_{i-1})-1}{2}$ good centers. 
Thus, the lemma is proven. 
\end{proof}

\begin{lem}\label{lem.q=2.lb}
For $ k\geq 2$, we have $\log_2(n)+1\leq \Z_2(C_{n,k}).$
\end{lem}

\begin{proof}
Observe that Blue must spend $1$ token in the first  turn of the $\Z_2$-forcing game on $C_{n,k}$. 
From this point on we will assume that White follows the protected path strategy. 
Notice that as long as $\varphi(H_i)>0$, there are still white vertices in $H_i$. 
Furthermore, by Lemmas \ref{lem.nodecrease} and \ref{lem.decrease}, Blue can only reduce $\varphi(H_i)$ by spending tokens. 
In particular, if Blue spends $1+i$ tokens to reach $H_j$, then $\varphi(H_j)\geq l_i$ where \[l_i = \frac{l_{i-1}-1}{2}\] and  $l_0=n-1$. 
We solve this recurrence relation in Lemma \ref{lemma:23} which can be found in the Appendix. The result is that $l_i\leq 0$ if and only if $i\geq \log_2(n)$. 
Thus, Blue must spend at least $\log_2(n)+1$ tokens to defeat White.
\end{proof}

We will now analyze the game from Blue's perspective to prove an upper bound. Consider the following strategy for Blue:
Let $c_0,\dots, c_{n-1}$ denote the center vertices on the cycle in $C_{n,k}$. To prove the upper bound, consider the following strategy for Blue.
First, Blue will spend $2$ tokens on leaves of adjacent centers. 
For the rest of the game, Blue has the following options.
\begin{enumerate}[label=Option\,\arabic*:, wide=0pt, leftmargin=*]

\item Apply Rule 2, if possible.
\item Apply Lemma \ref{force_lemma} and Rule 3, if possible. 
\item If there are fewer than 3 active vertices and a protected path $P$ with at least $3$ vertices, then spend a token on a leaf of an internal vertex of $P$.
\item Spend a token on a pendent vertex, if none of the other options are available. 
\end{enumerate}

\begin{lem}\label{lem.2paths}
Suppose that  $G$ has at least $2$ maximal protected paths of length at least 1 (at least 2 vertices in the path) and at least $3$ bad centers in $G$. Then Blue can turn a vertex blue for free. 
\end{lem}

\begin{proof}
Suppose that $G$ has $2$ maximal protected paths $P_1= c_0\cdots c_m$ and $P_2= c_s\cdots c_t$ where it is possible that $m=s$ or $t=0$. 
Since $G$ has at least $3$ bad centers, it cannot be the case that $m=s$ and $t=0$.
Without loss of generality, $t\neq 0$. 
For the sake of contradiction, suppose that Blue cannot turn a vertex blue for free. 
In particular, Blue cannot invoke Options 1 and 2. 
We will show that there are at least $3$ active vertices, which will be a contradiction by Lemma \ref{force_lemma}. 

Without loss generality, notice that if $c_0$ is white, then $c_0$ has a blue neighbor with degree $1$, or $c_{n-1}$ is blue and all the degree $1$ neighbors of $c_{n-1}$ are blue as well.
Since Blue cannot invoke Option 1, it follows that $c_0$ does not have any blue neighbor with degree 1. 
Therefore, $c_{n-1}$ is blue and all the degree $1$ neighbors of $c_n$ are blue as well. 
Once again, since Blue cannot invoke Option $1$, it follows that $c_{n-2}$ is white. 
Furthermore, $n-1\neq t$ since $c_t$ must have at least one white neighbor of degree $1$ by the definition of protected paths. 
In particular, $c_{n-1}$ is active. 

On the other hand, if $c_0$ is blue, then $c_0$ is active. 
Thus, we can conclude that either $c_0$ or $c_{n-1}$ is active. 

Notice that the same argument can be used to show that one of $c_m,c_{m+1}$ is an active vertex where we know that $c_0,c_{n-1},c_m, c_{m+1}$ are all distinct, and $m+1\neq t$.
This guarantees us 2 active vertices so far. 

We can also use the argument above to show that one of  $c_t,c_{t+1}$ is active. 
However, it is possible that $t+1= n-1$, and we do not increase our count of active vertices. 
Of course, if $t+1\neq n-1$ or $c_0$ is an active vertex, then we have found $3$ active vertices.
Therefore, we will proceed under the assumptions that $t+1=n-1$ and that $c_t,c_0$ are white  as we look for another active vertex.

Notice that if $s=m$, then there are at most two bad centers in $G$, which would contradict our assumption. 
Thus, $s\neq m$.

Now, let us consider the consequences of applying the argument above to conclude that one of $c_s,c_{s-1}$ is active. 
There are three cases: $c_s,c_m$ are both white, exactly one of $c_s,c_m$ is blue, and both $c_s,c_m$ are blue. 

Suppose $c_s$ and $c_m$ are white. 
Then either $c_{s-1}$ and $c_{m+1}$ are distinct active vertices, or there are only two bad centers in $G$.
This gives us our $3$ active vertices, or a contradiction; thus, the case when $c_s$ and $c_m$ are both white is resolved. 

Without loss of generality, suppose $c_m$ is blue and $c_s$ is white. 
Then $c_m$ and $c_{s-1}$ are both active. 
If $s-1=m$, then $G$ only has two bad vertices. 
Thus, we find 3 active vertices in $c_m,c_{n-1}, c_{s-1}$. 

Finally, if both $c_m$ and $c_s$ are blue, then we find $3$ active vertices in $c_{n-1},c_m,c_s$. 

In call cases, we are done. 
\end{proof}

\begin{lem}\label{lem.q=2.end}
If Blue cannot invoke Options 1,2, or 3 after spending at least $2$ tokens, then all the centers of $G$ are blue and at most two blue centers have white neighbors. 
\end{lem}

\begin{proof}
For the sake of contradiction, suppose a center $c_0$ is white. 
Since Blue cannot invoke Option $1$, it follows that all of $c_0$ degree 1 neighbors are white. Since the graph has at least two blue centers, there are distinct indices $i,j$ such that $c_i,c_j$ are blue and $c_k$ is white with only white degree $1$ neighbors  for $0< k < i$ and $j<k <n$.
This implies that $c_i\cdots c_j$ is a protected path, contradicting the fact that Option 3 cannot be invoked. 
Thus, all centers of $G$ are blue. 

Now suppose that there are three centers $x,y,z$ with white neighbors. 
Since Option 1 cannot be invoked, $x,y,z$ each have $2$ white neighbors, and are therefore, active.
Since there are $3$ active vertices, Blue could invoke Option 2; this is a contradiction. 
Thus,
at most two centers have white neighbors. 
\end{proof}

One consequence of Lemma \ref{lem.q=2.end} is that Blue will not use Option 4 until all centers are blue and at most two blue centers have white neighbors. 
Once the game reaches such a state, Blue is effectively playing to fill in a few leaves in order to close the game with some Rule 2 applications. 
The rest of the required analysis will focus on determining how often Blue  must invoke Option 3 to reach this final phase of the game.
The next lemma  states that if Blue invokes Option 3, then $\varphi$ should be reduced by a factor of $1/2$. 

\begin{lem}\label{lem.half}
Suppose there is exactly one protected path $P$ in $H_i$ such that $\varphi(H_i,P)>0$. 
Then Blue can spend a token in such a way that \[\varphi(H_{i+1})\leq \frac{\varphi(H_i)}{2}.\]
\end{lem}

\begin{proof}
Suppose that the protected $P$ is given by $c_1\cdots c_\ell$. 
Since $\varphi(H_i,P)>0$, it follows that $\ell\geq 3$. 
Consider $c_{x}$ where $x= \lceil\tfrac{\ell}{2}\rceil$. 
Since $c_{x}$ is on internal vertex on $P$, all the leaves on $c_{x}$ are white. 

Assume that Blue spends a token to color a degree 1 neighbor of $c_{x}$ blue in order to transition to $H_{i+1}$. 
This will split $P$ into two protected paths $P'=c_1\cdots c_x$ and $P''=c_x\cdots c_\ell$, the longer of which will realize $\varphi(H_{i+1})$. 
Notice that \[\varphi(H_{i+1},P') = x-2= \lceil\tfrac{k}{2}\rceil-2 \leq \frac{\varphi(H_i)+3}{2}-2=\frac{\varphi(H_i)-1}{2}\]
and 
\[\varphi(H_{i+1}, P'') = \ell-x-1 =\ell - \lceil\tfrac{\ell}{2}\rceil-1\leq \frac{\ell}{2}-1=\frac{\varphi(H_i)}{2}.\]
This concludes the proof.
\end{proof}

\begin{thm}\label{thm.q=2}
 For $ k\geq 2$ and $n\geq 3$, we have \[\log_2(n)+1\leq \Z_2(C_{n,k}) \leq \log_2(n-2) + 2k + 1.\]
\end{thm}

\begin{proof}
Blue spends $2$ tokens to start the game. 
Once Blue cannot invoke Option 1,2, or 3, all centers will be blue and at most two centers will have white neighbors. 
The worst case scenario for Blue, is that there are two centers with $k$ white neighbors of degree 1. 
Therefore, Blue will have to spend at most $2k-2$ tokens to finish the game. 

All that remains to be shown is that Blue will use Option $3$ at most $\log_2(n-2)+1$ times, since Blue will not use Option 4 until the final phase of the game. 
After Blue spends their first 2 tokens $\varphi(H_i) =  n-2$. 
Lemma \ref{lem.2paths} states that Blue will only spend  more tokens, if there is exactly one protected path.
Furthermore, once $\varphi(H_i)=0$, Blue will use Lemma \ref{lem.q=2.end} to reach the final phase of the game. 
Therefore, by Lemma \ref{lem.half}, Blue will spend at most $\log_2(n-2)+1$ tokens. 

Counting the total number of tokens Blue must spend in the worst case scenario gives the upper bound.
The lower bound was proven in Lemma \ref{lem.q=2.lb}.
\end{proof}

\subsection{Lower bound for $q=3$}\label{sec:q=3}

The proof of the lower bound for $\Z_3(C_{n,k})$ is similar to the proof of the lower bound for $\Z_2(C_{n,k})$, though there are a few more details to sort out. 
Some definitions are reintroduced in the specific context of $q=3$ so that this section is self contained. 
Since the proof of the lower bound for $q=3$ generalizes the proof of the lower bound for $q=2$, it is worth asking why we do not generalize the proof to large values of $q$. 
This question will be taken up at the end of the section, where the relevant details can be referenced. 

Suppose that $G$ is a graph isomorphic to $C_{n,k}$ with some number of blue vertices. 
We say a center $v$ is \textit{bad} if $v$ is blue or $v$ has at least one blue neighbor of degree one. 
Centers which are not bad are called \textit{good}.
We say  a path $P$ of  centers $c_1c_2\dots c_{m}$ is \textit{protected} if every vertex in $P$ has at least one white neighbor of degree one, there is at most one internal bad center in  $P$, no two bad centers in $P$ are adjacent, and $c_1,c_m$ are distinct and possibly bad. 
We call $c_2,\dots,c_{m-1}$ the \emph{internal} centers of $P$.

Let $\varphi(G,P)$ be the number of internal white vertices $w$ in $P$ such that all degree one neighbors of $w$ are white. 
In particular, $\varphi(G,P)$ is the number of good centers in $P$. 
Define $\varphi(G)= \max_P\{\varphi(G,P)\}.$

We will prove a lower bound on $\Z_{3}(C_{n,k})$ by analysing a strategy for White against arbitrary (optimal) play by Blue.
White will play according to the following strategy: 
If $\varphi(G)>0$, White will select a protected path of centers $P$ that realizes $\varphi(G)$. 
Whenever Blue uses Rule 3 and passes a set of white components $W=\{W_1, W_2, W_3, W_4\}$ to White, then White will do one of two things:
\begin{enumerate}[label=Option\,\arabic*:, wide=0pt, leftmargin=*]
    \item White passes back a component $W_i$ such that $W_i$ does not contain a vertex in $P$ or a degree one vertex adjacent to a vertex in $P$. 
    \item If every component in $W$ contains a vertex in $P$ or a degree one vertex adjacent to a vertex in $P$, then White passes back a subset $S\subseteq W$ such that Rule 2 cannot be used on $G[B\bigcup_{W_i\in S} W_i]$.
\end{enumerate}

\textbf{Claim:} If every component in $W=\{W_1, W_2, W_3, W_4\}$ contains a vertex in a protected path $P$ or a degree one vertex adjacent to a vertex in $P$, then there exists a subset $S\subseteq \{W_1, W_2, W_3, W_4\}$  such that Rule 2 cannot be used on $G[B\bigcup_{W_i\in S} W_i]$.
\begin{proof}

Without loss of generality, if $W_1$ and $W_2$ are isolated vertices adjacent to the same bad blue center in $P$, then using $S=\{W_1,W_2\}$ proves the claim. 
Therefore, we will assume that no two components in $W$ are isolated vertices adjacent to the same bad blue center of $P$. 
Now, if $P$ contains fewer than $3$ bad blue centers, then $W$ must contain two components that are isolated vertices adjacent to the same bad blue vertex (or $W$ contains a component which does not contain a vertex in $P$ or a degree one vertex adjacent to a vertex in $P$). 
Therefore, we will assume that $P$ has exactly $3$ bad blue centers.
By inspection, White can find the desired set $S$ (see Figure \ref{fig:findS}).
\end{proof}

\begin{figure}
    \centering
        \begin{tikzpicture}
        \draw[thick, gray!90]    (0,0) -- (1,0);
        \draw[thick, gray!90]    (2,0) -- (3,0);
        \draw[thick, gray!90]    (3,0) -- (4,0);
        \draw[thick, gray!90]    (5,0) -- (6,0);
        \draw[thick, dash dot, gray!90]    (1,0) -- (2,0);
        \draw[thick, dash dot, gray!90]    (4,0) -- (5,0);
        
        \draw[thick, gray!90]    (-.25,1) -- (0,0) -- (.25,1);
        \draw[thick, gray!90]    (.75,1) -- (1,0) -- (1.25,1);
        \draw[thick, gray!90]    (1.75,1) -- (2,0) -- (2.25,1);
        \draw[thick, gray!90]    (2.75,1) -- (3,0) -- (3.25,1);
        \draw[thick, gray!90]    (3.75,1) -- (4,0) -- (4.25,1);
        \draw[thick, gray!90]    (4.75,1) -- (5,0) -- (5.25,1);
        \draw[thick, gray!90]    (5.75,1) -- (6,0) -- (6.25,1);
        
        \draw [decorate, decoration = {calligraphic brace}] (-.3,1.2) --  (.3,1.2);
        \draw [decorate, decoration = {calligraphic brace}] (.7,1.2) --  (2.3,1.2);
        \draw [decorate, decoration = {calligraphic brace}] (2.7,1.2) --  (3.3,1.2);
        \draw [decorate, decoration = {calligraphic brace}] (3.7,1.2) --  (5.3,1.2);
        \draw [decorate, decoration = {calligraphic brace}] (5.7,1.2) --  (6.3,1.2);
        
        \node[circle, label={90:$1$}] at (0,1.2) {};
        \node[circle, label={90:$2$}] at (1.5,1.2) {};
        \node[circle, label={90:$3$}] at (3,1.2) {};
        \node[circle, label={90:$4$}] at (4.5,1.2) {};
        \node[circle, label={90:$5$}] at (6,1.2) {};

        \filldraw[blue] (0,0) circle (2pt);
        \filldraw[white] (1,0) circle (2pt);
        \draw[black] (1,0) circle (2pt);
        \filldraw[white] (2,0) circle (2pt);
        \draw[black] (2,0) circle (2pt);
        \filldraw[blue] (3,0) circle (2pt);
        \filldraw[white] (4,0) circle (2pt);
        \draw[black] (4,0) circle (2pt);
        \filldraw[white] (5,0) circle (2pt);
        \draw[black] (5,0) circle (2pt);
        \filldraw[blue] (6,0) circle (2pt);
        
        \filldraw[white] (-.25,1) circle (2pt);
        \draw[black] (-.25,1) circle (2pt);
        \filldraw[white] (.25,1) circle (2pt);
        \draw[black] (.25,1) circle (2pt);
        
        \filldraw[white] (.75,1) circle (2pt);
        \draw[black] (.75,1) circle (2pt);
        \filldraw[white] (1.25,1) circle (2pt);
        \draw[black] (1.25,1) circle (2pt);
        
        \filldraw[white] (.75,1) circle (2pt);
        \draw[black] (.75,1) circle (2pt);
        \filldraw[white] (1.25,1) circle (2pt);
        \draw[black] (1.25,1) circle (2pt);
        
        \filldraw[white] (1.75,1) circle (2pt);
        \draw[black] (1.75,1) circle (2pt);
        \filldraw[white] (2.25,1) circle (2pt);
        \draw[black] (2.25,1) circle (2pt);
        
        \filldraw[white] (2.75,1) circle (2pt);
        \draw[black] (2.75,1) circle (2pt);
        \filldraw[white] (3.25,1) circle (2pt);
        \draw[black] (3.25,1) circle (2pt);
        
        \filldraw[white] (3.75,1) circle (2pt);
        \draw[black] (3.75,1) circle (2pt);
        \filldraw[white] (4.25,1) circle (2pt);
        \draw[black] (4.25,1) circle (2pt);
        
        \filldraw[white] (4.75,1) circle (2pt);
        \draw[black] (4.75,1) circle (2pt);
        \filldraw[white] (5.25,1) circle (2pt);
        \draw[black] (5.25,1) circle (2pt);
        
        \filldraw[white] (5.75,1) circle (2pt);
        \draw[black] (5.75,1) circle (2pt);
        \filldraw[white] (6.25,1) circle (2pt);
        \draw[black] (6.25,1) circle (2pt);

        \end{tikzpicture}
    \caption{We have assumed that $W$ does not contain two components from zones $1$, $3$, or $5$. 
    If $W$ contains components from zones $\{1,2,3\}$, $\{3,4,5\}$, or $\{1,2,4,5\}$, then White can pass back those components. 
    Notice that any set of four components will fall into at least one of these three cases.}
    \label{fig:findS}
\end{figure}
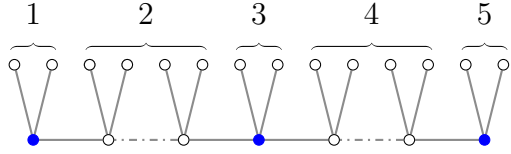

Suppose $\{H_i\}_{i=0}^t$ is a sequence of colored copies of $C_{n,k}$ where $H_0$ has all white vertices, $H_t$ has all blue vertices, and $H_i$ follows from $H_{i-1}$ by one application of some $\Z_3$ forcing rule. 
In particular, assume that White plays according to the strategy above in the sequence $\{H_i\}_{i=1}^t$. 
We will also assume Blue successfully turns a vertex blue with every move so that $H_{i}$ has exactly one more blue vertex than $H_{i-1}$.

\begin{lem}\label{lem.noprogress}
If Blue does not spend a token to transition from $H_{i-1}$ to $H_i$, then $\varphi(H_i)= \varphi(H_{i-1})$. 
\end{lem}

\begin{proof}

Let $P$ be a protected path in $H_{i-1}$ that realizes $\varphi(H_{i-1})$. 
For the sake of contradiction, assume that  Blue did not spend a token to transition from $H_{i-1}$ to $H_i$ and  $\varphi(H_i)< \varphi(H_{i-1})$. 
Let $w$ be a white vertex in $H_{i-1}$ but blue in $H_i$.
Since Blue did not spend a token, Blue must have used either Rule 2 or Rule 3. 
Since $\varphi(H_i)< \varphi(H_{i-1})$, we know that $w$ is a vertex in $P$, or $w$ is a degree one neighbor of a vertex in $P$.

Suppose that Blue used Rule 3 to transition from $H_{i-1}$ to $H_i$.
Since White was following the strategy given above, we know that White passed Blue some component in $W_i\in \{W_1, W_2,W_3,W_4\}$ such that $W_i$ does not contain a vertex in $P$ or a degree one vertex adjacent to a vertex in $P$. 
Therefore, Blue cannot use Rule 3 to turn a vertex in $P$ blue. 
Furthermore, Blue cannot use Rule 3 to turn a degree one neighbor of a vertex in $P$ blue.
Thus, we will assume that Blue used Rule 2 to transition from $H_{i-1}$ to $H_i$.

Suppose $w$ is a degree one neighbor of a vertex in $P$ where $w$ is white in $H_{i-1}$ but blue in $H_{i}$.
This implies that $w$ is adjacent to a bad center $v$ of $P$ in $H_{i-1}$. 
However, since $P$ is a protected path, all bad centers of $P$ have at least $2$ white neighbors. 
In particular, Blue cannot use Rule 2 to turn $w$ blue (ruling out a case).

Suppose $w\in V(P)$ is white in $H_{i-1}$ but is blue in $H_i$. 
This implies that there exists a blue vertex $v$ in $H_{i-1}$ which is adjacent to $w$. 
If $v$ is a vertex of degree one, then $w$ was a bad center in $H_{i-1}$ and $P$ is still protected in $H_i$. 
If $v$ is a center, then $v\notin V(P)$ since every vertex in $P$ has at least one white neighbor of degree one. 
However, this implies that $w\in \{c_1,c_m\}$ and, therefore, $P$ is protected in $H_i$. 
Since $P$ is protected in $H_i$ whether or not $v$ is a center, it follows that $\varphi(H_i)\geq \varphi(H_{i-1})$; a contradiction.  
\end{proof}

We will continue our analysis by distinguishing between two kinds of protected paths.
Notice that all protected paths have either $2$ or $3$ bad centers.

\begin{lem} \label{lem.2protected}
If Blue spends a token to transition from $H_{i-1}$ to $H_i$ and $\varphi(H_{i-1})$ is realized by a protected path with $2$ bad centers, then $\varphi(H_i)\geq \varphi(H_{i-1})-1$.
\end{lem}

\begin{proof}
Suppose $P$ is a protected path with $2$ bad centers that realizes $\varphi(H_{i-1})$. 
If $\varphi(H_{i-1})=0$, then we are done. 
Therefore, assume that $\varphi(H_{i-1})>0$. 
Turning one vertex blue will create at most one bad center in $P$ in $H_i$.
In particular, $P$ has at least $\varphi(H_{i-1})-1$ good centers in $H_i$. 
Since $P$ only has $2$ bad centers in $H_{i-1}$, it follows that $P$ has at most $3$ bad centers in $H_i$ and that $P$ is protected in $H_i$. 
\end{proof}

\begin{lem}\label{lem.2protected.extension}
Suppose $H_{i-1}$ has a protected path with  $\varphi(H_{i-1},P)$ good centers.
If Blue spends one token to transition from $H_{i-1}$ to $H_i$, then $H_i$ has a protected path $P'\subseteq P$ with at least $\tfrac{\varphi(H_i,P)-1}{3}$ good centers and 2 bad centers. 
In particular, 
\[\varphi(H_i)\geq \tfrac{\varphi(H_{i-1})-1}{3}. \]
\end{lem}

\begin{proof}
After Blue spends a token, $P$ will have at least $g-1$ good centers divided over at most $3$ subpaths of good centers. 
\end{proof}

\begin{lem}\label{lem.2step}
If Blue spends 2 tokens to transition from $H_{i}$ to $H_j$ and $\varphi(H_{i})$ is realized by a protected path with $2$ bad centers, then $H_j$ has a protected path with $2$ bad centers and  $\tfrac{\varphi(H_{i})-2}{3}$ good centers. In particular,
\[\varphi(H_j)\geq \frac{\varphi(H_{i})-2}{3}.\]
\end{lem}

\begin{proof}
By Lemmas \ref{lem.noprogress}, $\varphi$ may decrease only during the time steps in which Blue spends a token. 
Suppose Blue spends a token at time steps $i_1$ and $i_2$, with $i<i_1<i_2\leq j$. 
By Lemma \ref{lem.2protected}, 
\[\varphi(H_{i_1})\geq \varphi(H_{i})-1.\]
By Lemma \ref{lem.2protected.extension}, 
\[\varphi(H_{i_2})\geq \tfrac{\varphi(H_{i_1})-1}{3}.\]
Substituting the first inequality into the second gives the result. 
\end{proof}

\begin{thm}\label{thm.corona.lower}
Suppose $q=3$, $k\geq 2$, and $n\geq 3$. Then \[2\log_3(2)\log_2(\tfrac{n}{200}) \leq \Z_{3}(C_{n,k}).\]
\end{thm}

\begin{proof}
Let $G_0=H_{i_0}$ be the first term in the sequence $\{H_i\}_{i=1}^t$ with two bad centers. 
Observe that $\varphi(G_0)\geq \tfrac{n-2}{2}=\varphi _0$. 
Let $G_j =H_{i_j}$ be a term in the sequence such that Blue spends $2$ tokens to transition to $H_{i_j}$ from $G_{j-1}$. 
By Lemma \ref{lem.2step}, 
\[\varphi(G_j)\geq \frac{\varphi(G_{j-1})-2}{3}\geq \varphi_j\]
where $\varphi_j$ satisfies the recurrence relation $\varphi_j=\tfrac{\varphi_{j-1}-2}{3}$.
Solving the recurrence relation using generating functions shows that \[\varphi(G_j)\geq \frac{n-2}{2\cdot 3^j}+\frac{1}{ 3^{j}}-1= \frac{n}{2\cdot 3^j} - 1.\]
Notice that as long as $\varphi(G_j)\geq 99$, White has not lost.
Therefore, Blue must spend at least $2\log_3(2)\log_2(\tfrac{n}{200})\approx 1.26 \log_2(\tfrac{n}{200})$ tokens. 
\end{proof}

Generalizing the proof for larger values of $q$ seems cumbersome since larger values of $q$ will allow for more bad centers within protected paths. 
Allowing more bad centers within a protected path gives Blue more choices of how to break up a protected path. 
For example, Blue may choose to spend many tokens to split a protected path.
Furthermore, these choices affect the distribution of bad centers within protected paths at later stages of the game.  
These considerations suggest that generalization is hard. 

Furthermore, a back of the envelope calculation suggests that if the proof of Theorem \ref{thm.corona.lower} is generalized for $q>3$, then the leading constant will be $(q-1)\log_q(2)$.
This value is significantly smaller than the upper bound we derive in Subsection \ref{sec:upperbound}.

\subsection{Upper bound for $q\geq 3$.}\label{sec:upperbound}

We are able to crudely generalize the upper bound proof from Theorem~\ref{thm.q=2}. 
Our goal is to understand how $\Z_q(C_{n,k})$ depends on $n$, $q$, and $k$. 
First, we reduce Blue's goal to only coloring the centers of $C_{n,k}$. 
Then we provide a strategy that Blue can use to turn all centers blue. 
Though we still consider maximal paths of white centers (as we did with protected paths), we no longer wish to track which centers are good or bad. 
This imprecision will manifest itself in a slightly worse upper bound in terms of $q$ and $k$. 

\begin{lem}\label{lem.q.geq.3.end}
Suppose $G$ is a copy of $C_{n,k}$ with $n\geq 3$ and $k\geq 2$ such that the blue set of $G$ contains all the centers. 
Then Blue can turn $G$ blue by spending at most $k-q+q(k-1)$ additional tokens. 
\end{lem}

\begin{proof} 
Let $B$ be the blue set of $G$.
By Observation \ref{edge_lemma}, $\Z_q(G,B)= \Z_q(nS_k, B)$. 
The result follows from Proposition~\ref{prop.stars}.
\end{proof}

In general, we will suppose that $n$ is large. 
Consider the following strategy Blue can use to turn all  centers blue:
Blue will spend $p=q+1+x$ tokens where $x\equiv q+1\mod 2$ to color $p/2$ pairs of adjacent centers blue such that these pairs are as equally spaced out as possible on the cycle in $C_{n,k}$. Then Blue has the following options.

\begin{enumerate}[label=Option\,\arabic*:, wide=0pt, leftmargin=*]
\item Apply Lemma \ref{force_lemma} and Rule 3, if possible. 
\item Apply Rule 2, if possible.
\item If there are  at most $p/2 - 1$ maximal paths of white centers, then Blue will spend at most $2$ tokens within each maximal path of white centers to split each path into two paths of equitable length (or color the path completely blue). 
\end{enumerate}

Suppose $\{H_i\}^t_{i=0}$ is a sequence of colored copies of $C_{n,k}$ where $H_0$ has all white vertices, $H_t$ has all blue vertices, and $H_i$ follows form $H_{i-1}$ by one application of some $Z_q$ forcing rule; however, for the simplicity of notation, we will assume that Option 3 can be executed in one time step. 
In particular, assume that Blue plays according to the strategy above in the sequence $\{H_i\}_{i=1}^t$. 

\begin{obs}\label{obs.bookends}
Suppose $H_i$ is a game state after Blue has spent the first $p$ tokens. 
Then every blue center in $H_i$ is adjacent to another blue center in $H_i$. 
\end{obs}

\begin{lem} 
If Blue cannot invoke Options 1 or 2, then Blue can invoke Option 3. 
\end{lem}

\begin{proof}
Assume that the blue set $B$ of $H_i$ precludes Blue from invoking Options 1 or 2. 
For the sake of contradiction, suppose $H_i$ has at least $p/2$ maximal paths of white centers. 
Let $A$ be the set of blue centers that are adjacent to a white center which is an endpoint of a maximal path of white centers.
By Observation \ref{obs.bookends}, $|A|\geq p\geq q+1$. 

Suppose $v$ is in $A$. 
This implies that $v$ has at least $1$ white neighbor. 
Since Option 1 cannot be invoked, $v$ must be adjacent to a white vertex of degree $1$. 
Therefore, $v$ is blue with at least $2$ white neighbors, and is considered active. 
This implies that $A$ is a set of at least $q+1$ active vertices, each of which is adjacent to a white degree $1$ vertex. 
This contradicts the fact that Option 2 cannot be invoked. 
Thus, 
$H_i$ has fewer than $p/2$ maximal paths. 
\end{proof}

\begin{lem}
Blue will use Option 3 at most $\log_2(n)$ times. 
\end{lem}

\begin{proof}
After the Blue spends the first $p$ tokens, we have that the the longest path of white centers has at most \[\left\lceil\frac{2(n-p)}{p}\right\rceil\leq \frac{n}{2}\]
vertices. 
Each time Blue invokes Option 3, the length of the longest path is reduced by a factor of at least $1/2$. 
Therefore, if $j$ is the number of times Blue invokes Option 3, 
then 
\begin{align*}
    \frac{n}{2}\frac{1}{2^{j-1}}=\frac{n}{2^{j}}&\geq 1\\
    n&\geq 2^{j}\\
    \log_2(n)&\geq j.
\end{align*}
This concludes the proof of the lemma.
\end{proof}

\begin{thm}
\label{thm.corona.upper}
 Let $n\geq 3$, $k,q\geq 2$, and $p= q+1+x$ where $x\equiv q+1\mod 2$. 
 Then \[2\log_3(2)\log_2(\tfrac{n}{200}) \leq \Z_{q}(C_{n,k}) \leq (p-2)\log_{2}(n) + p+ k-q+q(k-1).\]
\end{thm}

\begin{proof}
Blue will follow the strategy given in this section to turn all the centers of $C_{n,k}$ blue.
This will require at most 
$2(p/2-1)\log_2(n)+p$ 
tokens. 
After all the centers are Blue, Lemma \ref{lem.q.geq.3.end} guarantees that at most $k-q+q(k-1)$ tokens are required for Blue to finish the game.
\end{proof}

\begin{cor} \label{cor.cat.upper}
For the caterpillar graph $P_{n,k}$,
\[2\log_3(2)\log_2(\tfrac{n}{200}) -2\leq \Z_q(P_{n,k}) \leq (p-2)\log_{2}(n) + p+ k-q+q(k-1)+1\]
where $p$ is defined as in Theorem \ref{thm.corona.upper}.
\end{cor}

\begin{proof}
Apply Theorem \ref{thm.edge.deletion} to Theorem \ref{thm.corona.upper}.
\end{proof}

Since $P_n\circ kK_1$ is a tree, Theorem \ref{thm.tree.upper} also applies. 
However, a simple comparison shows that Corollary \ref{cor.cat.upper} is a significant improvement. 
Suppose $n$ is odd and $q\geq 2$.
Then 
\[\min_{v\in V(P_n\circ kK_1)}\left\{\deg(v) + \sum\limits_{l=1}^{h-1} \sum\limits_{a\in A_{\ell} } \mathop{\max} \{ 0,\deg(a)-2\}\right\}\geq k+2 + \frac{n-3}{2}2k. \]
Notice $k+2 + \tfrac{n-3}{2}2k$ dominates $(p-2)\log_{2}(n) + p+ k-q+q(k-1)$ asymptotically in both $n$ and $k$.

\section*{Acknowledgements}
This research was conducted in Summer 2022 as a part of the Summer Undergraduate Applied Mathematics Institute at Carnegie Mellon University.

\bibliographystyle{plainurl}
\bibliography{bibliography}

\section{Appendix}

\begin{lem} \label{lemma:23}
Let the sequence $\{l_i\}_{i=0}^\infty$ be defined by the following recurrence relation
\[
l_i = \frac{l_{i-1}}{2} - \frac{1}{2}\quad \text{and} \quad l_0 = n-1.
\]
For any $i\geq 0$, 
\[
     l_i = (l_0 + 1) \left(\frac{1}{2}\right)^i -1.
\]
Moreover $ 0\geq l_j $  if and only if $j\geq \log_2(n)$. 
\end{lem}

\begin{proof}
We show the first result by a generating functions argument. Consider: 
\begin{align*}
    F(x) &= \sum_{i=0}^\infty l_ix^i \\
    &= l_0 +\frac{1}{2} xF(x) - \frac{1}{2}\left( \frac{1}{1-x} - 1 \right).
\end{align*}
Next we solve for $F(x)$ we obtain that: 
\begin{align*}
    F(x) &= \frac{l_0}{\left(1-\frac{x}{2}\right)} - \frac{1}{2(1-x)\left(1-\frac{x}{2}\right)}+\frac{1}{2\left(1-\frac{x}{2}\right)}\\
    F(x) &= \sum_{i=0}^{\infty} \left( (l_0 + 1) \left(\frac{1}{2}\right)^i -1  \right) x^i\\
\end{align*}
Therefore,
\begin{align*}
    l_i = (l_0 + 1) \left(\frac{1}{2}\right)^i -1
\end{align*} 
for $i\geq 0$
as desired.
Since $l_0 = n-1$, we have that 
$ 0\geq l_j $  if and only if $j\geq \log_2(n)$. 
\end{proof}

\end{document}